\documentclass[letterpaper, 10 pt, conference]{ieeeconf}

\IEEEoverridecommandlockouts 
\usepackage{graphicx} 
\usepackage{cite}
\usepackage{amsmath,amssymb,amsfonts}
\usepackage{url}
\usepackage{color, soul}
\usepackage{hyperref}
\usepackage{graphicx}
\usepackage{caption}
\usepackage{subcaption}
\usepackage{soul}
\usepackage{optidef}
\usepackage{listings}
\usepackage{amstext}
\usepackage{verbatim}
\usepackage{indentfirst}
\usepackage[utf8]{inputenc}
\usepackage{caption}
\usepackage{subcaption}
\usepackage{algorithm}
\usepackage{framed}
\usepackage{rotating}
\usepackage{cite}
\usepackage{optidef}
\usepackage{graphicx}
\usepackage{booktabs}
\usepackage{algpseudocode}
\algrenewcommand\algorithmicrequire{\textbf{Data:}}
\algrenewcommand\algorithmicensure{\textbf{Return:}}

\newcommand{\real}{\mathbb{R}}
\newcommand\abf{\mathbf{a}}
\newcommand\bbf{\mathbf{b}}

\newcommand\gbf{\mathbf{g}}
\newcommand\ibf{\mathbf{i}}

\newcommand\pbf{\mathbf{p}}

\newcommand\rbf{\mathbf{r}}

\newcommand\vbf{\mathbf{v}}

\newcommand\xbf{\mathbf{x}}
\newcommand\ybf{\mathbf{y}}
\newcommand\zbf{\mathbf{z}}

\newcommand\lambdabf{\boldsymbol{\lambda}}

\newtheorem{theorem}{Theorem}
\newtheorem{lemma}{Lemma}
\newtheorem{definition}{Definition}

\newtheorem{proposition}{Proposition}

\newtheorem{remark}{Remark}

\graphicspath{{/IMG/}} 

\begin{document}
\title{\LARGE \bf  Learning-Based Stackelberg Equilibrium Seeking with Application to Demand-Side Energy Management}

\author{Silvia Cianchi, Reza Rahimi Baghbadorani, Anibal Sanjab, Sergio Grammatico \thanks{ Silvia Cianchi is with Delft University of Technology, and VITO and EnergyVille  \href{mailto:s.cianchi@tudelft.nl}{\texttt{s.cianchi@tudelft.nl}}. Reza Rahimi Baghbadorani is with Erasmus University \href{mailto:rahimibaghbadorani@rsm.nl}{\texttt{rahimibaghbadorani@rsm.nl}}. Anibal Sanjab is with VITO and EnergyVille \href{mailto:anibal.sanjab@vito.be}{\texttt{anibal.sanjab@vito.be}}. Sergio Grammatico is with Delft University of Technology \href{mailto:s.grammatico@tudelft.nl}{\texttt{s.grammatico@tudelft.nl}}.}}
\maketitle

\begin{abstract}
Demand-side management (DSM) enables distribution system operators (DSOs) to steer electricity consumption through dynamic price signals or incentive mechanisms, thereby leveraging end-users' flexibility potential for delivering grid services. The resulting hierarchical interaction between the DSO and the end-users can be formulated as a Stackelberg game, where the operator dynamically sets the prices and the end-users optimally respond to them. Efficiently designing these price signals is challenging, as the users’ response models are unknown or difficult to estimate. In this paper, we propose a learning-based zeroth-order algorithm for incentive design, in which the iterative update of the incentive signals is efficiently assisted by a data-driven online estimation of the users' responses. The proposed method is then proven to converge to an equilibrium tariff while allowing the DSO to estimate the decision-making problems at the user level. Moreover, the method preserves users' privacy, as the update rule of the DSO is solely based on observations of communicated end-user actions. Numerical simulations employing real-world data illustrate the efficient convergence of our learning-based proposed method, while significantly reducing the number of required interactions between the DSO and the end-users with respect to the state-of-the-art approach.
\end{abstract}

\section{Introduction}
The increasing penetration of distributed energy resources and storage enables end-users to dynamically generate, store, and actively manage their electricity consumption. This inherent flexibility forms the basis for demand-side management (DSM) \cite{albadi2008summary}, whereby distribution system operators (DSOs) design price signals and incentive mechanisms to unlock this flexibility, thus ensuring grid services needs are met while delivering benefits to end-users. The interaction between DSOs and end-users can be modeled as a hierarchical decision-making process, in which DSOs dynamically set price signals to ensure reliable operation of the grid and recuperate their investment costs, and 
end-users, in turn, react to these imposed tariffs and accordingly schedule their energy demand in order to optimize their objective (e.g., minimizing costs and discomfort) \cite{cianchi2025two}. 
This interconnected setting is typically modeled as a Stackelberg game \cite{von2010market}, with the DSO being the leader (solving the upper-level problem, setting the price signals) and the end users as followers (solving the lower level-problem, capturing their optimal reaction to the communicated price signal). The key challenge in solving a Stackelberg game lies in how the leader accounts for the reaction of the followers. In fact, if the DSO has enough information about the decision-making process at the lower level, then it can anticipate/predict the reactions of the end-users to centrally solve the problem and set the tariffs. However, in practice, such information is not readily available (posing a setting of incomplete information), thereby requiring the DSO to iteratively communicate with the end-users, to converge to an equilibrium price setting. A challenge, therein, however, is that an increasing number of iterations between the DSO and end-users to reach convergence limits the practical deployment of such techniques. In this paper, we propose a methodology which allows a DSO to \emph{efficiently} design an effective tariff scheme by iteratively querying the end-users and learning their decision-making process.

\smallskip
\noindent
\emph{Related Work.}
In recent years, DSM mechanisms have been widely investigated in the literature. The works in \cite{9993196, askeland2023stochastic, koolen2017machine, guo2019drivers} assume that the DSO has access to a model of the decision-making process of the end-users, which can be either given \cite{9993196, askeland2023stochastic}, or approximated offline based on historical data \cite{koolen2017machine, guo2019drivers}. In these cases, the DSO can embed the end-users' reactions in its optimization problem and centrally solve a bilevel optimization problem, typically reformulated as 
a Mathematical Programming with Equilibrium Constraints (MPEC) \cite{luo1996mathematical}. This approach exhibits several limitations, as it heavily relies on information about end-users' behavior, which is often unavailable and uncertain. As such, errors in the lower-level model can propagate into the upper-level decision of the DSO, and cannot be corrected online.
Moreover, this approach suffers from a computational burden, as MPECs are often intractable, especially for realistic distribution network models \cite{luo1996mathematical}. To solve these issues, works such as \cite{zhang2013robust, nguyen2017optimal, cianchi2025two, 10886163, alahmed2024network} develop iterative algorithms, either based on dual decomposition schemes \cite{zhang2013robust, nguyen2017optimal, cianchi2025two} or on problem-specific solutions \cite{10886163, alahmed2024network}, in which the DSO repeatedly updates the network tariffs after collecting local information from the end-users. These approaches work for rather specific formulations, such as under the assumption that the leader’s objective is the sum of the users’ objectives plus a global convex term, which makes the problem separable and convex. 
Moreover, the first-order iterative approach proposed in \cite{grontas2024big} suffers from an information burden, as the end-users are supposed to repeatedly communicate their responses to the tariffs, as well as the Jacobian of the responses, which might violate privacy requirements, and induces additional complexity at the users' end. This limitation is successfully overcome in \cite{maheshwari2024follower}, where the leader, by probing the followers with different incentives, computes a zeroth-order estimation of its gradient. However, this technique requires a significant number of interactions of the DSO with the end-users, which may hinder its practical efficiency. Finally, reinforcement learning and bandit methods have been applied to design DSM mechanisms \cite{hutchinson2024safe, lu2018dynamic}, where the DSO interacts with end-users to learn price strategies without requiring detailed models of their decision making processes. However, these schemes typically need a large amount of training data, and rely on exploration, which may induce significant variability in the iterates, potentially leading to disruptive and undesirable price changes. 

\smallskip
\noindent
\emph{Paper Contributions.}
In this paper, we propose a learning-based approach for dynamic tariff design that significantly improves efficiency by enabling the DSO to leverage a learned model of end-users' behavior. Our method builds on the zeroth-order method proposed in \cite{maheshwari2024follower}, which, by iteratively querying the end-users, thereby preserving their privacy, computes stable updates, but converges slowly. By efficiently integrating the learned model of the end users' reaction, our proposed framework retains the advantages of the zeroth-order updates, while significantly reducing the number of required interactions. Firstly, we provide a revisited version of the zeroth-order method proposed in \cite{maheshwari2024follower} to extend it to our particular DSM problem, that is nonsmooth. In this regard, we provide convergence guarantees to a stationary point of the Stackelberg game as well as an explicit convergence bound. Secondly, we provide a novel framework to efficiently integrate the learned-model of the end-users' reactions with the zeroth-order updates. Finally, we corroborate the effectiveness of our proposed approach through simulations based on real-world data, illustrating the performance in realistic operating conditions. In our numerical experience, our proposed learning-based approach is several orders of magnitude faster than the standard zeroth-order methods, thus significantly reducing the number of interactions between the DSO and the end-users, which is key in practical implementations where the frequency of the communications between the operator and distributed agents is practically constrained.

\smallskip
\noindent
\emph{Paper Organization.}
The paper is organized as follows: Section \ref{Formulation} introduces the problem formulation and the solution concept. Section \ref{Methodology} presents the proposed methodology. Section \ref{NumRes} presents the simulation results, and Section \ref{Conclusion} concludes the paper.

\smallskip
\noindent
\emph{Notation.}
We let $\xbf=[x_1, ... , x_n] \in \real^n$ denote a vector, and $x_t$ be its $t$-th component. We express as $\boldsymbol{0}_n$ and $ \mathbf{1}_n$ vectors with all zero and all one elements, respectively. The operator $\mathrm{col}(b_1,\dots,b_n)$ indicates the column vector obtained by stacking the elements $b_1,\dots,b_n$. Furthermore, we denote by $\mathrm{proj}_C(\xbf):= \underset{{ \ybf \in C }}{\mathrm{argmin}} ||\xbf - \ybf||^2 $ the Euclidean projection operator onto a convex set $C$, and by $\mathcal{B}(\zbf, \rho)$ a ball centered in $\zbf$ with radius $\rho$. Finally, $\mathrm{VI}(F,\Omega)$ defines a variational inequality (VI) problem, whose set of solutions is represented as $\mathrm{SOL}(F,\Omega)=\left\{\xbf^* \in \Omega\,|\,F(\xbf^*)^\top(\xbf-\xbf^*)\ge0, \forall \xbf \in \Omega)\right\}$.
\section{Problem formulation}\label{Formulation}
We consider a DSO setting the network
tariff $\ybf \in \real^T$ for the following day ($T=24$), and an energy community (EC) which, by reacting to it, adjusts its energy demand.
In particular, the EC is composed of $N$ prosumers aiming at scheduling their electricity demand over the given time horizon. Let agent $i$ be equipped with $m_i$ assets, and let $\xbf_{i,t} \in \real^{m_i}$ denote the vector containing the consumption level of each asset of agent $i$ at time slot $t$. Moreover, let us define $\xbf_i:=\mathrm{col}((\xbf_{i,t})_{t=1,\dots,T})$, and $\xbf:=\mathrm{col}((\xbf_{i})_{i=1,\dots,N})$, with $\xbf \in \real^d, \,d=\sum_{i=1}^Nm_i T$. For any fixed $\ybf$, we model the interactions among the users in the EC as a Nash game with coupling constraints:
\begin{subequations}\label{Followers}
\begin{align}
\forall i: \underset{\xbf_i \in \mathcal{X}_i} {\min} \, &J_i(\xbf_i; \ybf):=||\xbf_i - \rbf_i||^2 + (\pbf+\ybf)^\top \Pi_i A_i \xbf_i, \label{FollowersObj}
\\
\mathrm{s.t.}\, \, &\sum_{i=1}^N A_i\xbf_i \le \bbf. \label{FollowersConstr}
\end{align}
\end{subequations} 

In particular, each prosumer schedules their consumption in order to optimize their individual objective function, \eqref{FollowersObj}, which is a (weighted) tradeoff between discomfort and costs. For each agent, we define the discomfort as the Euclidean distance from a desired consumption profile $\rbf_i \in \real^{m_iT}$ which reflects the preferences set by the agent. In addition, each energy user is subject to a cost for its total energy offtake, consisting of its consumption level, $A_i \xbf_i$, multiplied by the dynamic retail price $\pbf$ $\in \real^T$ and the network tariff $\ybf \in \real^T$ set by the DSO. The objective function of each user is also characterized by a specific weighting matrix, $\Pi_i$, which can model a variety of scenarios, such as per-device responses to price, behavioral preferences, user- (or even device-)specific contractual rules. For example, if each user specifies its sensitivity to applied electricity prices, $\pi_i$, that is, its willingness to deviate from $\rbf_i$ (increasing discomfort) to achieve economic savings, then the weighting matrix would be $\Pi_i=\pi_i I_T$. Moreover, $\mathcal{X}_i$ denotes the set of local constraints of each prosumer, capturing the technical constraints of the assets (e.g., power and energy limits, state of charge evolution, etc.) as well as behavioral preferences (e.g., satisfaction of the demand, desired state of charge, etc.). Detailed characterization of these constraints is provided in the numerical simulations in Section \ref{NumRes}. Furthermore, we consider that the EC aims to guarantee, thanks to the presence of a community manager, the safe operation of the grid by satisfying the grid capacity limits in \eqref{FollowersConstr}, imposed by the grid operator, via, e.g., connection capacity agreements.  We highlight that the cumulative constraint in \eqref{FollowersConstr} can model cumulative loading constraints imposed via financial contracts as well as physical capacity limits imposed by the grid operator \cite{SWEDEN}. 
As a solution concept, we adopt the notion of generalized Nash equilibrium (GNE):
\begin{definition}{\textbf{Generalized Nash equilibrium:}}
For a given $\ybf$, a solution $\left(\xbf^*_1, ..., \xbf^*_N\right)$ of the game in \eqref{Followers} is a generalized Nash equilibrium (GNE) if
\begin{equation}\label{NEWGAME}
\begin{aligned}
      \xbf^*_i \in \quad &\underset{\xbf_i \in \mathcal{X}_i} {\arg \min} \, J_i(\xbf_i; \ybf)\\ 
     &\mathrm{s.t.} \, A_i \xbf_i + \sum_{j=1, j \ne i}^N  A_j\xbf^*_j  \le  \bbf. \nonumber
\end{aligned}
\end{equation}
\end{definition}
Specifically, we investigate a specific subset of generalized Nash equilibria, that is the set of \emph{variational} generalized Nash equilibria ($v$-GNE), in view of their well-established variational interpretation \cite{facchinei2003finite}.
\begin{definition}{\textbf{Variational generalized Nash equilibrium:}}\label{vGNE}
For a given $\ybf$, a solution $\left(\xbf^*_1, ..., \xbf^*_N\right)$ of the game in \eqref{Followers} is a $v$-GNE if it is a solution of the following VI: $\mathrm{VI}(F(\cdot \, ;\ybf), \Omega)$, where $F(\xbf;\ybf):=\mathrm{col}\left((\nabla_{\xbf_i} J_i(\xbf_i; \ybf))_{i =1,...,N}\right)$, and $\Omega:=\{\xbf \in \mathcal{X}_1\times... \times \mathcal{X}_N| \sum_{i=1}^N A_i\xbf_i \le \bbf\}.$
\end{definition}
For simplicity of notation, we denote the set of $v$-GNEs of the followers' game, given $\ybf$, as $\mathcal{E}(\ybf):=\mathrm{SOL}(F(\cdot\,;\ybf),\Omega)$. 
\begin{lemma}\label{lemma1}
    The set of $v$-GNEs of the game in \eqref{Followers} is a singleton, $\mathcal{E}(\ybf)=\left\{\xbf^*(\ybf)\right\}$, where
    \begin{align}\label{proj}
        x^*(\ybf):=\mathrm{proj}_C(\Theta_A\ybf-\Theta_b),
    \end{align}
    with $\Theta_A=\mathrm{col}\left((\Theta_{A_i})_{i=1,\dots,N}\right)$, ${\Theta_A}_i=-\frac{1}{2} A_i^\top\Pi_i^\top$, and $\Theta_b=\mathrm{col}\left((\Theta_{b_i})_{i=1,\dots,N}\right)$, ${\Theta_b}_i=A_i^\top \Pi_i^\top\pbf - \rbf_i$. Also, $C=\Omega$ includes both shared and local constraints and is a polyhedron, i.e., $C=\left\{\xbf \in \real^d | G \xbf \le \mathbf{h}\right\}$.  \hfill $\blacksquare$
\end{lemma} 
 Let us postulate that the DSO designs the network tariff with the aim of recovering its (regulated) investment budget, as in \eqref{Leader}. Since the DSO moves first (setting the network tariffs) and the prosumers then react to that tariff, we cast this hierarchical interconnected decision making process among the DSO and the end-users in the EC as a Stackelberg game, where the DSO acts as the leader, and the energy users as the followers:
\begin{subequations}\label{Stackelberg}
\begin{align}
\underset{\ybf \in \mathcal{Y}} {\min} \, \,&J_0(\ybf,\xbf):=-\ybf^\top \Lambda A \xbf+ \mu \left(\mathbf{1}^\top\ybf-\bar{y}\right)^2+\lambda\left\|\ybf\right\|^2\label{Leader}
\\
\mathrm{s. t.}\,&  \xbf \in \mathcal{E}(\ybf), \label{LeaderConstr}
\end{align}
\end{subequations}
where $\mathcal{Y}=\real^T$, and $A=\left[A_1, \dots, A_N\right]$. In particular, the DSO designs the network tariff, $\ybf$, in order to maximize the economical revenue, while ensuring that the tariff remains consistent with regulatory or contractual requirements. Specifically, the revenue is computed as the aggregative hourly energy offtake of the EC, $A \xbf$, weighted by a matrix $\Lambda$, modeling a variety of scenarios (e.g., contractual rules), multiplied by the network tariff $\ybf$. The term $\mu\left(\mathbf{1}^\top\ybf-\bar{y}\right)^2$ enforces consistency with a prescribed average tariff $\bar{y}$, while $\lambda \left\|\ybf\right\|^2$ penalizes excessively large and volatile tariffs.   
In view of Lemma \ref{lemma1}, the parametric VI constraint in \eqref{LeaderConstr} can be replaced by the single-valued solution mapping $\xbf^*(\cdot)$ in \eqref{proj}. As such, the bilevel problem in \eqref{Stackelberg} can be recast as an unconstrained nonsmooth (nonconvex) optimization problem:
\begin{align}\label{NonConvex}
\underset{\ybf \in \mathcal{Y}} {\min} &\, J_0(\ybf, \xbf^*(\ybf)). 
\end{align}

In the remainder of the paper, we refer to $J_0(\cdot,\xbf^*(\cdot))$ as $J_0(\cdot)$ for ease of notation. 
The global minima of \eqref{NonConvex} are the so-called Stackelberg equilibria (SE) \cite{dempe2002foundations}. However, since the bilevel optimization problem in \eqref{NonConvex} is, in general, nonconvex, we extend the solution set of the game to the set of stationary points of \eqref{NonConvex} \cite{grontas2024big, maheshwari2024follower}. 
 
\section{Learning and seeking}\label{Methodology}
In this section, we propose a novel methodology that enables the DSO to efficiently reach a stationary point of \eqref{NonConvex}. We consider a setup in which the DSO has no direct access to the internal decision-making process of the prosumers and can interact with the EC only using queries. In departure from standard black-box settings, we assume that the DSO knows the parametric structure of the lower-level reaction, denoted as
\begin{align}\label{Param}
\xbf^*(\ybf; \Theta_A, \Theta_b):=\mathrm{proj}_C(\Theta_A \ybf-\Theta_b),
\end{align}
where the parameters $\Theta_A$ and $\Theta_b$ are unknown to the DSO. In this section, we adopt the simplifying assumption that the DSO knows the feasible set $C$ of the EC defined in Lemma \ref{lemma1}. 
However, in Section \ref{NumRes} we empirically show that this assumption can be relaxed by allowing the DSO to iteratively update an estimate of $C$ using online-collected data.

Our main contribution lies in proposing a hybrid approach that effectively assists zeroth-order iterations with parameter learning to efficiently solve \eqref{NonConvex}. The procedure consists of the following repeated steps:
\begin{enumerate}[label=\alph*)]
\item The DSO queries the prosumers and adjusts the network tariff using a zeroth-order update rule. The responses of the EC to each proposed tariff, $\left (\ybf_k, \xbf^*(\ybf_k)\right)$, are collected in a dataset;
\item After $K$ iterations, the DSO estimates the unknown parameters $(\hat{\Theta}_A, \hat{\Theta}_b)$ of the lower-level decision process \eqref{Param} based on the collected dataset;
\item Using this estimate $(\hat{\Theta}_A, \hat{\Theta}_b)$, the DSO computes an approximate solution $\hat{\ybf}$ of \eqref{NonConvex} and restarts the next round of zeroth-order interactions with $\hat{\ybf}$ as the initial condition.
\end{enumerate}

In the following, we describe each step in detail. We note that while our zeroth-order updates step builds on the existing method in \cite{maheshwari2024follower} with minor modifications, the key novelty of our approach is the systematic integration of learning with zeroth-order iterations to improve convergence.
\subsection{Zeroth-order algorithm} 
We propose a zeroth-order algorithm that enables the DSO to iteratively adjust the tariff $\ybf$ by relying solely on observed responses from the prosumers in the EC, consisting of their consumption levels. Our method builds upon the zeroth-order optimization technique in \cite[Algorithm 1]{maheshwari2024follower}. However, in our framework, the problem \eqref{NonConvex} is nonsmooth due to the presence of the projection operator in $x^*(\cdot)$, which violates the smoothness assumption required by the method in \cite{maheshwari2024follower}.
To address this, we consider a smooth approximation of \eqref{NonConvex}
\begin{align}\label{J0beta}
    \tilde{J}_0(\ybf):=J_0(\ybf, \xbf^{\beta}(\ybf)), 
\end{align}
obtained by replacing $\xbf^*(\cdot)$ with a smooth version, $\xbf^{(\beta)}(\cdot)$ \cite{nesterov2005smooth}. Specifically, the hard constraints of the projection are replaced with a smooth squared softplus penalty weighted by a smoothing parameter $\beta >0$ \cite[Eq. (2)]{li2025new}: 
 \begin{align*}
 \xbf^{(\beta)}(\ybf) := \arg \min\limits_{\zbf \in \real^T}\,&{\frac{1}{2} \left\|\zbf-\ybf\right\|^2 + \beta \sum_{i=1}^{q}(\mathrm{sftpl}(G_i\zbf-h_i))^2},
 \end{align*}
 where $\mathrm{sftpl}(t)=\log(1+\mathrm{e}^{\frac{t}{\beta}})$. 
 Notably, it holds that stationary points of \eqref{J0beta} converge to Clarke stationary points of the original nonsmooth function \eqref{NonConvex} as $\beta \to 0^+$ \cite[Theorem 1]{nesterov2005smooth}\cite[Proposition 2.2.2]{clarke1990optimization}. Consequently, stationary points of \eqref{J0beta} provide arbitrarily accurate approximations to the solutions of the original game in \eqref{NonConvex}.
 Let us denote the Lipschitz and smooth constants of $\tilde{J}_0(\cdot)$ as $\tilde{L}$ and $\tilde{\ell}$, respectively.
 
 As such, we propose Algorithm \ref{alg:ZOM} which enables the DSO to solve \eqref{J0beta}. At each iteration $k$, the DSO communicates to the EC the price signal, $\ybf_k$, along with a perturbed version of it, $\tilde{\ybf}_k=\ybf_k+\delta_k\vbf_k$, where $\tilde{\ybf}_k$ is a point uniformly randomly picked within a ball $\mathcal{B}(\ybf_k, \delta_k)$. Based on the corresponding reactions of the prosumers, i.e., $\xbf^*(\ybf_k),$ and $ \xbf^*(\tilde{\ybf}_k)$, we let the DSO build the gradient estimator
\begin{align}\label{Est}
     \hat{\gbf}_k=\frac{T}{\delta_k}(J_0(\ybf_k)- \,J_0(\ybf_k+\vbf_k\delta_k))\vbf_k,
 \end{align}
 where, as defined in Section \ref{Formulation}, $J_0(\cdot)=J_0(\cdot, x^*(\cdot))$.
 Next, we let the DSO update the tariff scheme as
 \begin{align}\label{Update}
    \ybf_{k+1}=\ybf_k-\eta_k\,\hat{\gbf}_k.
\end{align}

We highlight that the classical zeroth-order two-point gradient estimator \cite[Eq. (30)]{nesterov2017random} for the smoothed problem in \eqref{J0beta} should, in principle, evaluate the smoothed reactions, $\xbf^{(\beta)}(\cdot)$, of the energy users, and take the form
\begin{align}\label{estimator}
     \gbf_k=\frac{T}{\delta_k}(\tilde{J}_0(\ybf_k)- \,\tilde{J}_0(\ybf_k+\vbf_k\delta_k))\vbf_k.
 \end{align}
In practice, however, this is not feasible, since the DSO can only observe the actual user responses, $\xbf^*(\ybf_k)$, $\xbf^*(\tilde{\ybf}_k)$. As such, the estimator $\hat{\gbf}_k$, while formally targeting the smoothed problem in \eqref{J0beta}, operates on the nonsmooth responses \eqref{proj} instead. This discrepancy induces an additional error in the gradient estimator $\hat{\gbf}_k$, which represents the main difference between our approach and that in \cite{maheshwari2024follower}.
 \begin{algorithm}
\caption{Zeroth-order algorithm to solve \eqref{J0beta}}\label{alg:ZOM}
\begin{algorithmic}[1]
   \Require Number of iterations $K$, Initial condition $\ybf_0 \in \real^T,$ Step sizes $(\eta_k)_{k \in [K]}$, Perturbation radius $(\delta_k)_{k \in [K]}$ 
   \For{$k = 1$ to $K-1$}
   \State Sample $\vbf_k \sim \mathrm{Unif}(\mathcal{S}(\real^T))$;
   \State Assign $\tilde{\ybf}_k=\ybf_k + \vbf_k \delta_k$;
   \State Collect $\xbf_k=\xbf^*(\ybf_k)$;
   \State Collect $\tilde{\xbf}_k=\xbf^*(\tilde{\ybf}_k)$;
   \State Compute $\hat{\gbf}_k=\frac{T}{\delta_k}(J_0(\ybf_k,\xbf_k)- \,J_0(\tilde{\ybf}_k,\tilde{\xbf}_k))\vbf_k$;
   \State Update $\ybf_{k+1}=\ybf_k-\eta_k \hat{\gbf}_k$;
   \EndFor
\end{algorithmic}
\end{algorithm}
We now establish the convergence of Algorithm \ref{alg:ZOM} in Theorem \ref{TheoBound}. 
\begin{theorem}\label{TheoBound}
    Let $\eta_k=\bar{\eta}(k+1)^{-1/2}T^{-1}, \delta_k=\bar{\delta}(k+1)^{-1/4}T^{-1/2}$, such that $\bar{\eta}\le T/2 \tilde{\ell}$. Then, the iterations $\left(\ybf_k\right)_{k \in [K]}$ of Algorithm \ref{alg:ZOM} are such that
    \begin{align}\label{Bound}
        \frac{1}{\sum_{k=1}^K \eta_k} \sum_{k=1}^K \eta_k \mathbb{E}\left[\left\|\nabla \tilde{J}_0(\ybf_k)\right\|^2\right]\le \mathcal{U}(\tilde{J}_0(\ybf_0),K,\bar{\eta}, \bar{\delta}),
    \end{align}
    where the expectation is taken with respect to $\vbf_k$, and $\mathcal{U}=\mathcal{U}_1 + \mathcal{U}_2+\mathcal{U}_3+\mathcal{U}_4$, with $\mathcal{U}_1=\frac{T (\tilde{J}_0(\ybf_0)-\tilde{J}_0^*)}{\bar{\eta}\sqrt{K}};\, \mathcal{U}_2=\frac{\bar{\ell}^2T \bar{\delta}^2 \mathrm{log}(K)}{4 \sqrt{K}};\, \mathcal{U}_3=\mathcal{O}(\sqrt{K}\beta^2);\,\mathcal{U}_4=\frac{4T\bar{L}^2\bar{\ell}\bar{\eta}\mathrm{log}(K)}{\sqrt{K}}$, assuming that $\tilde{J}_0(\ybf)\ge\tilde{J}^*_0 \,\,\forall \, \ybf \in \real^T$.
\end{theorem}
\begin{proof}
    The proof builds on the proof of \cite[Theorem 1]{maheshwari2024follower}: by providing a bound on the bias terms of the gradient estimator $\hat{\gbf}_k$, we show that $\tilde{J}_0(\ybf_k)$ approximately decreases over the successive iterations $\left(\ybf_k\right)_{k \in [K]}$.
\end{proof}
The convergence gap in \eqref{Bound} quantifies the performance of Algorithm \ref{alg:ZOM} in approaching a stationary point of \eqref{J0beta}. In particular, the components $\mathcal{U}_1, \, \mathcal{U}_2$, and $\mathcal{U}_4$ in \eqref{Bound} decrease with the number of iterations, $K$, reflecting the fact that longer runs reduce the errors and drive the iterates closer to a stationary point of \eqref{J0beta}. This behavior is standard in zeroth-order methods. In contrast, the term $\mathcal{U}_3$, that is induced by using the estimator $\hat{\gbf}_k$ in \eqref{Est}, instead of $\gbf_k$ in \eqref{estimator} (i.e., evaluating the original nonsmooth objective, $J_0$, instead of its smoothed approximation, $\tilde{J}_0$) grows with $K$. Nevertheless, this contribution can be made arbitrarily small by selecting a sufficiently small smoothing parameter $\beta$.

Furthermore, per each iteration of Algorithm \ref{alg:ZOM}, we let the DSO collect two input-output data points, $\left(\ybf_k,\xbf^*(\ybf_k)\right)$, $\left(\tilde{\ybf}_k,\xbf^*(\tilde{\ybf}_k)\right)$, and include them in the dataset $\mathcal{D}:=\left\{(\ybf_k, \xbf_k)\right\}_k$, which is then used in the next step. 
\subsection{Estimation of the parameters}
After $K$ iterations of Algorithm \ref{alg:ZOM}, the dataset $\mathcal{D=}\left\{\left(\ybf_k, \xbf_k\right)\right\}_{k=1}^{2K}$, containing the proposed tariffs and the corresponding observed user responses, has been collected by the DSO. Using this dataset, we let the DSO compute an estimation of the parameters of the decision-making process of the prosumers, $(\hat{\Theta}_A, \hat{\Theta}_b)$. This estimation step is a key component of our hybrid approach: in departure from standard zeroth-order methods that treat the lower-level as a black box, we here let the DSO exploit the estimated parametric model \eqref{Param} to improve the convergence performance of Algorithm \ref{alg:ZOM}. This is done by minimizing the empirical suboptimality loss:

\begin{align}\label{Loss}
    \min\limits_{\hat{\Theta}_A, \hat{\Theta}_b} & \frac{1}{2K}\sum_{k=1}^{2K}{\left\|J_0(\ybf_k, \xbf_k)-J_0(\ybf_k,\xbf^*(\ybf_k;\hat{\Theta}_A, \hat{\Theta}_b))\right\|^2}\\
    \text{s.t.} \quad &\xbf^*(\ybf_k; \hat{\Theta}_A, \hat{\Theta}_b)=\mathrm{proj}_\mathcal{C}(\hat{\Theta}_A\ybf_k-\hat{\Theta}_b).\nonumber
\end{align}

Theorem \ref{Surrogate} captures the local validity of the estimated model \eqref{Param} derived by the DSO, with $\mathcal{B}(\ybf_k,\rho)$ as its trust region.
\begin{theorem}\label{Surrogate}
    For each data point $(\ybf_k,\xbf_k) \in \mathcal{D}$, there exist $\epsilon>0, \, \rho >0$ such that:
    \begin{itemize}
        \item For all $\ybf \in \mathcal{B}(\ybf_k, \rho)$, it holds that \begin{align}
        \left\|\xbf^*(\ybf)-\xbf^*(\ybf; \hat{\Theta}_A, \hat{\Theta}_b)\right\|\le \epsilon;
        \end{align}
        \item Let $L$ be the Lipschitz constant of $J_0(\ybf,\xbf)$ with respect to $\xbf$, and $\hat{\ybf}^* \in \arg \min\limits_{\ybf \in \mathcal{B}(\ybf_k, \rho)}\, J_0(\ybf, \xbf^*(\ybf; \hat{\Theta}_A, \hat{\Theta}_b))$. Then it holds that 
        \begin{align}\label{jump}
        J_0(\hat{\ybf}^*, \xbf^*(\hat{\ybf}^*))\le J_0(\ybf_k, \xbf^*(\ybf_k))+ 2 L \epsilon.
        \end{align}
    \end{itemize}
\end{theorem}
\begin{proof}
    The proof can be derived by exploiting the fact that, at each data point, the suboptimality loss is equal to zero, and by combining the continuity of $x^*(\cdot)$, the Lipschitzness of $J_0$, and the definition of $\hat{y}^*$.
\end{proof}
One additional key novelty of our approach lies in exploiting the estimation $\xbf^*(\ybf; \hat{\Theta}_A, \hat{\Theta}_b)$ to solve the original problem \eqref{NonConvex} constrained within the trust region $\mathcal{B}(\ybf_k,\rho)$, to find a candidate point $\hat{\ybf}^*$ that is a warm start for Algorithm \ref{alg:ZOM}. This update is designed to further decrease the true objective function, $J_0(\cdot, \xbf^*(\cdot))$, before restarting the next round of Algorithm \ref{alg:ZOM}. The motivation comes from the fact that the convergence bound \eqref{Bound} of Algorithm \ref{alg:ZOM} depends on the distance between the objective value at the initial point and at its minimum. If a reduction is not achieved, the condition in \eqref{jump} ensures that any deterioration in the objective function remains controlled and bounded.
\subsection{Solution of the estimated Stackelberg game}
As a third step, we let the DSO exploit the learned lower level model, $x^*(\ybf; \hat{\Theta}_A, \hat{\Theta}_b)$, to derive a suitable candidate point, $\hat{\ybf}^*$, for updating the current tariff $\ybf_K$, reached after $K$ iterations of Algorithm \ref{alg:ZOM}. Specifically, we propose that the DSO computes the candidate point by centrally solving the bilevel optimization problem in \eqref{NonConvex} with the lower-level responses replaced by the estimated model $x^*(\ybf;\hat{\Theta}_A, \hat{\Theta}_b)$. To this end, we employ Algorithm \ref{alg:DCP}, which leverages the difference-of-convex programming (DCP) approach \cite{tao1997convex}. The structure of the bilevel problem \eqref{NonConvex} naturally admits a DCP reformulation, as stated in the following proposition.
\begin{proposition}\label{DCPRef}
    The problem in \eqref{NonConvex} can be equivalently reformulated as the constrained optimization of a difference of convex functions, as follows:
    \begin{align}\label{DCPFunc}
        \min\limits_{\zbf \in \mathcal{Z}}\quad & g(\zbf)-p(\zbf),
    \end{align}
    where $\zbf=\mathrm{col}\left(\xbf, \ybf, \mathbf{\theta}, \zeta\right)$ and $\mathcal{Z}=\left\{(\xbf,\ybf,\theta,\zeta)\,|\,\xbf-(\Theta_A \ybf-\Theta_b)+G^\top \mathbf{\theta}=0 \,, \, G \xbf \le \mathbf{h} \,, \, \mathbf{\theta} \ge 0 \right\}$, with $\mathbf{\theta}$ and $\zeta$ being the dual variable associated with the feasible set $C$ and a fictitious variable, respectively. In addition, $g(\zbf)=\mu (\mathbf{1}^\top\ybf-\bar{\ybf})^2+\lambda \left\|\ybf\right\|^2 + \zeta \mathbf{\theta}^\top\mathbf{h}+\frac{1}{4} \left\|-A^\top \Lambda^\top \ybf - \xbf \right\|+ \frac{\zeta}{4} \left\|G^\top \mathbf{\theta} - \xbf\right\|^2$, and $p(\zbf)= \frac{1}{4} \left\|-A^\top \Lambda^\top \ybf + \xbf \right\|+ \frac{\zeta}{4} \left\|G^\top \mathbf{\theta} + \xbf\right\|^2$. \hfill $\blacksquare$
\end{proposition}
 In view of Proposition \ref{DCPRef}, Algorithm \ref{alg:DCP} can be read as the DCP algorithm applied to the problem in \eqref{DCPFunc}, where, at each iteration $t$, the concave part $-p(\zbf)$ is linearized, turning the problem into a convex optimization problem. Under standard assumptions, Algorithm \ref{alg:DCP} converges to a stationary point of the objective function, and the generated sequence $\left(g(\zbf_t)-p(\zbf_t)\right)_{t \in [t_{\max}]}$ is monotonically decreasing \cite[Theorem 3]{tao1997convex}.
\begin{algorithm}
\caption{Difference of Convex Programs \cite[Eq. (11)]{tao1997convex}} \label{alg:DCP}
\begin{algorithmic}[1]
   \Require Initial condition $\zbf_0$
   \While{$t \le t_{\max}$}
   \State Solve $\zbf_{t+1} \in \arg \min\limits_{\zbf \in \mathcal{Z} }\quad {g(\zbf)- \langle \nabla p(\zbf_t),\zbf\rangle}$
   \EndWhile
   \Ensure $\hat{\ybf}=\ybf_{t_{\max}}$
\end{algorithmic}
\end{algorithm}

In Algorithm \ref{alg:DCP}, the tuning of the stopping criterion $t_{\max}$ is crucial. Results from Theorem \ref{Surrogate} ensure that the estimated lower-level model \eqref{Param} is accurate only within a trust region $\mathcal{B}(\ybf_{K},\bar{\rho})$. Consequently, the iterates of Algorithm \ref{alg:DCP}, initialized at $\zbf_0 = \mathrm{col}\left(\ybf_{K}, \xbf_{K}, \lambdabf_0\right)$, are meaningful approximations of the original problem \eqref{Stackelberg} as long as they remain within this region. For a given $\bar{\rho}$, it is therefore essential to select $t^*_{\max} \geq 0$, such that the iterates of Algorithm \ref{alg:DCP} do not leave the trust region, i.e., $\left(\ybf_t\right)_{t \in [t^*_{\max}]} \subseteq \mathcal{B}(\bar{\rho}, \ybf_{K})$. However, explicitly characterizing this region is generally intractable. Thus, we propose to approximate $t^*_{\max}$ as the largest iteration index $t$ for which the sequence generated by Algorithm \ref{alg:DCP} yields a monotonic decrease of the true objective, i.e., $\forall\, t=1,\dots, t^*_{\max}$ it holds
\begin{align}\label{Crit}
    J_0(\ybf_t, x^*(\ybf_t)) < J_0(\ybf_{t-1}, x^*(\ybf_{t-1})).
\end{align}

This criterion serves as a practical proxy to ensure that the iterates remain within the region where the estimated model is reliable.
Yet, in practice, the DSO does not have direct access to $\xbf^*(\cdot)$ (which is based on the true parameters $(\Theta_A, \Theta_b)$), and thus cannot determine $t^*_{\max}$ a priori (without probing the end-users). Instead, it must empirically choose $t_{\max}$ and validate the resulting candidate solution $\hat{\ybf}$ by querying the prosumers and evaluating the resulting objective.

If the candidate solution $\hat{\ybf}$ provides a satisfactory reduction in the objective function $J_0$, then we let the DSO start a new round of Algorithm \ref{alg:ZOM} with $\hat{\ybf}$ as the initial condition. In the following remark, a criterion to evaluate the quality of the computed candidate point $\hat{\ybf}$ is provided.

\begin{remark}
Let us assume that Algorithm \ref{alg:ZOM}, initialized in $\ybf_0$, runs for $K_1$ iterations; the estimation step, followed by Algorithm \ref{alg:DCP}, returns $\hat{\ybf}$, and, by querying the end-users, the DSO evaluates $J_0(\hat{\ybf})$. By exploiting the convergence bound in \ref{Bound}, the DSO can derive the minimum number of iterations, $K_2$, of the next round of Algorithm~\ref{alg:ZOM} such that accepting the candidate point $\hat{\ybf}$ is beneficial for convergence. The associated adaptation rule is the following:
\begin{align}\label{AdaptiveN}
    \mathcal{U}(\tilde{J}_0(\hat{\ybf}), K_2, \bar{\eta}', \bar{\delta}')<\mathcal{U}(\tilde{J}_0(\ybf_0), K_1+K_2, \bar{\eta}, \bar{\delta}), 
\end{align}
where the left-hand side is the bound obtained by restarting Algorithm \ref{alg:ZOM} from the estimated $\hat{\ybf}$ with smaller $\eta_k, \delta_k$, while the right-hand side is the bound obtained by discarding $\hat{\ybf},$ and letting the current round of Algorithm \ref{alg:ZOM} run for $K_2$ more iterations. The value of $\beta$ can be chosen arbitrarily small, hence, the function evaluations $\tilde{J}_0(\hat{\ybf})$ and $\tilde{J}_0(\ybf_0)$ in \eqref{AdaptiveN} can be accurately approximated by $J_0(\hat{\ybf})$ and $J_0(\ybf_0)$ \cite[Proposition 2.1]{li2025new}.
\end{remark}
\section{Numerical results}\label{NumRes}
As a numerical test case, we consider the setup in \cite{cianchi2025two}, consisting of a EC where each energy user $i$ can be equipped with: \begin{itemize}
    \item A load $\xbf^\text{L}_i \in \real^T$, which is partially shiftable and flexible.
    \item A battery, which can be charged and discharged at rates of, respectively, $\xbf_i^\text{BAT}$ and $\abf^\text{BAT}_i + \ibf^\text{BAT}_i \in \real^T$. In particular, $\abf_i^\text{BAT}$ is the power used by the consumer for its own demand, and $\ibf_i^\text{BAT}$ is the power that is injected into the grid.
    \item An electric vehicle (EV), which is charged at a rate of $\xbf^\text{EV}_i \in \real^T$ during a specific time window  $[\underline{h}, ..., \bar{h}]$.
    \item A solar panel (PV), whose generated power can be used, ($\abf_i^\text{PV} \in \real^T$), injected, ($\ibf_i^\text{PV} \in \real^T$), or curtailed.
\end{itemize}

Let us denote the state of user $i$ as 
\begin{align}
    \xbf_i=\mathrm{col}\left(\xbf_i^\text{L}, \xbf_i^\text{BAT}, \xbf_i^\text{EV}, \abf_i^\text{BAT}, \abf_i^\text{PV}, \ibf_i^\text{BAT}, \ibf_i^\text{PV}\right),
\end{align}
and the total power offtake and injection of consumer $i$ are respectively formulated as:
\begin{equation}
\begin{aligned}
    &A_i \xbf_i = \xbf_i^\text{L} + \xbf_i^\text{BAT} + \xbf_i^\text{EV} - \abf_i^\text{BAT} - \abf_i^\text{PV},\\
    &\tilde{A}_i \xbf_i = \ibf_i^\text{BAT} + \ibf_i^\text{PV}.
\end{aligned}
\end{equation} 

The local feasibility set of each energy user, $\mathcal{X}_i$, includes the satisfaction of the demand and the operational constraints of each device, as in \cite[Eq. (11)]{cianchi2025two}. 

Furthermore, we consider that each energy user can specify (e.g., via a home energy management system) its behavioral preferences by setting suitable entries for the desired profile $\rbf_i$, and the weighting matrix $\Pi_i=\pi_iI_T$.
We run the simulations for $N=10$ energy users with real data and parameters as in \cite{cianchi2025two}. 

Moreover, we choose $\Lambda=I_{T}$, $\mu=1000$, $\lambda=50$. \textcolor{black}{Also, the optimal value $J^*$ of \eqref{NonConvex} is computed using the well-known nonlinear solver \texttt{Ipopt}\footnote{\url{https://coin-or.github.io/Ipopt/}}}.
\begin{figure}
        \includegraphics[width=\columnwidth]{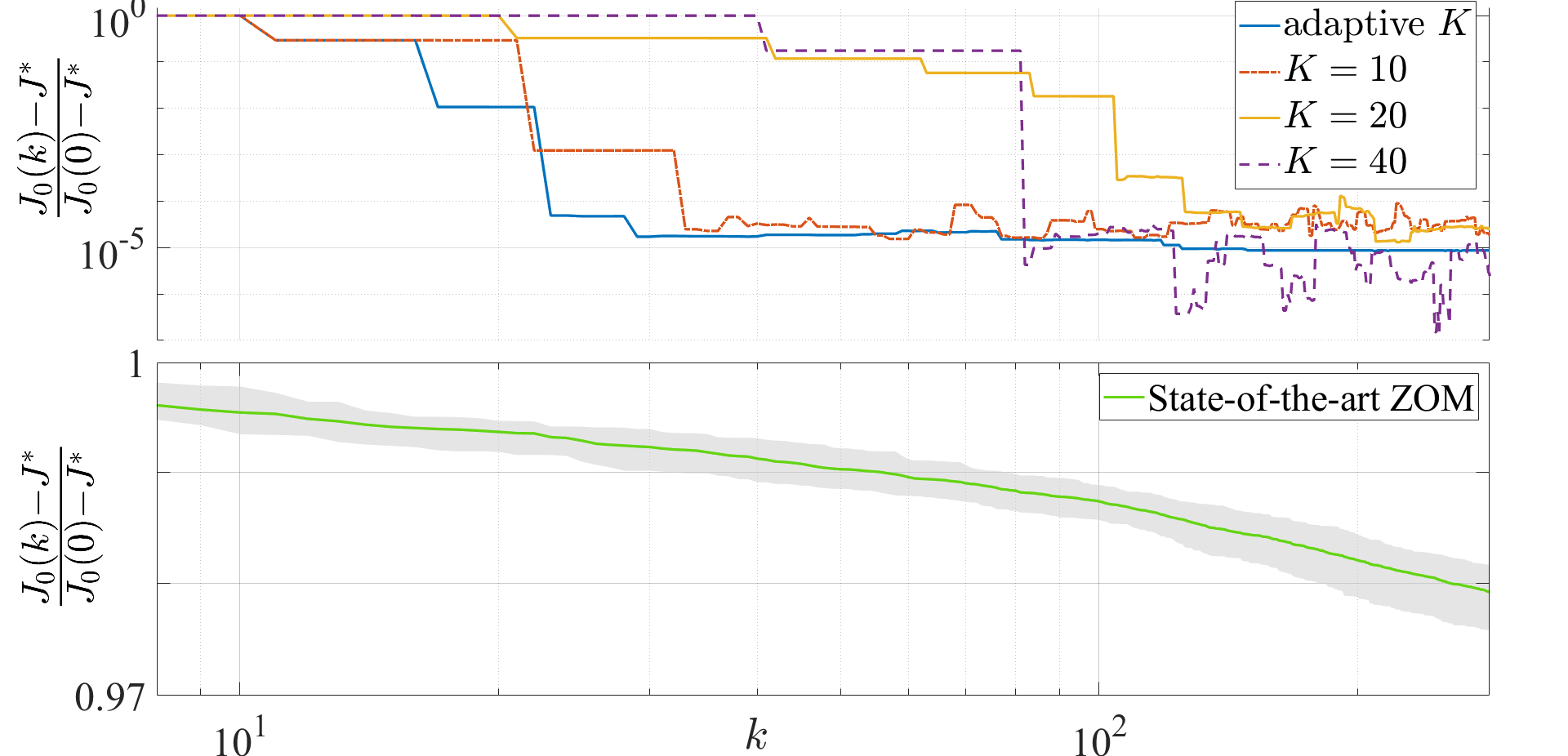}
        \captionsetup{width=\columnwidth}
        \caption{Convergence of the DSO objective, $J_0$, for different values of $K$, with the optimal $t^*_{\max}$, and the true feasible set, $C$. At the bottom, a comparison with Algorithm \ref{alg:ZOM}, when it is not assisted by the estimation of the user response.}
        \label{FIG1}
\end{figure}

\begin{figure}
        \includegraphics[width=\columnwidth]{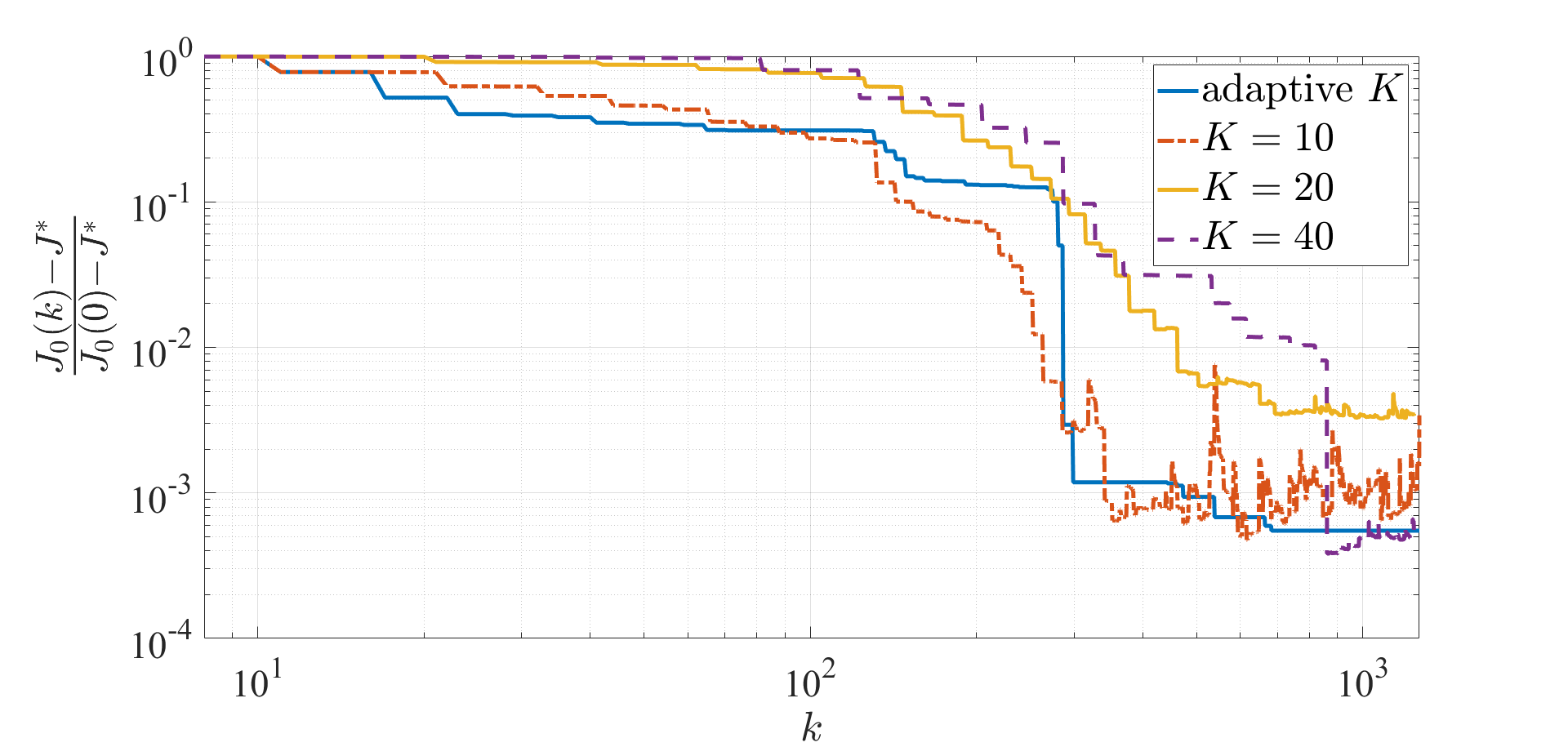}
        \captionsetup{width=\columnwidth}
        \caption{Convergence of the DSO objective, $J_0$, for different values of $K$, with the optimal $t^*_{\max}$, and learned feasible set.}\label{FIG2}
\end{figure}

\begin{figure}
        \includegraphics[width=\columnwidth]{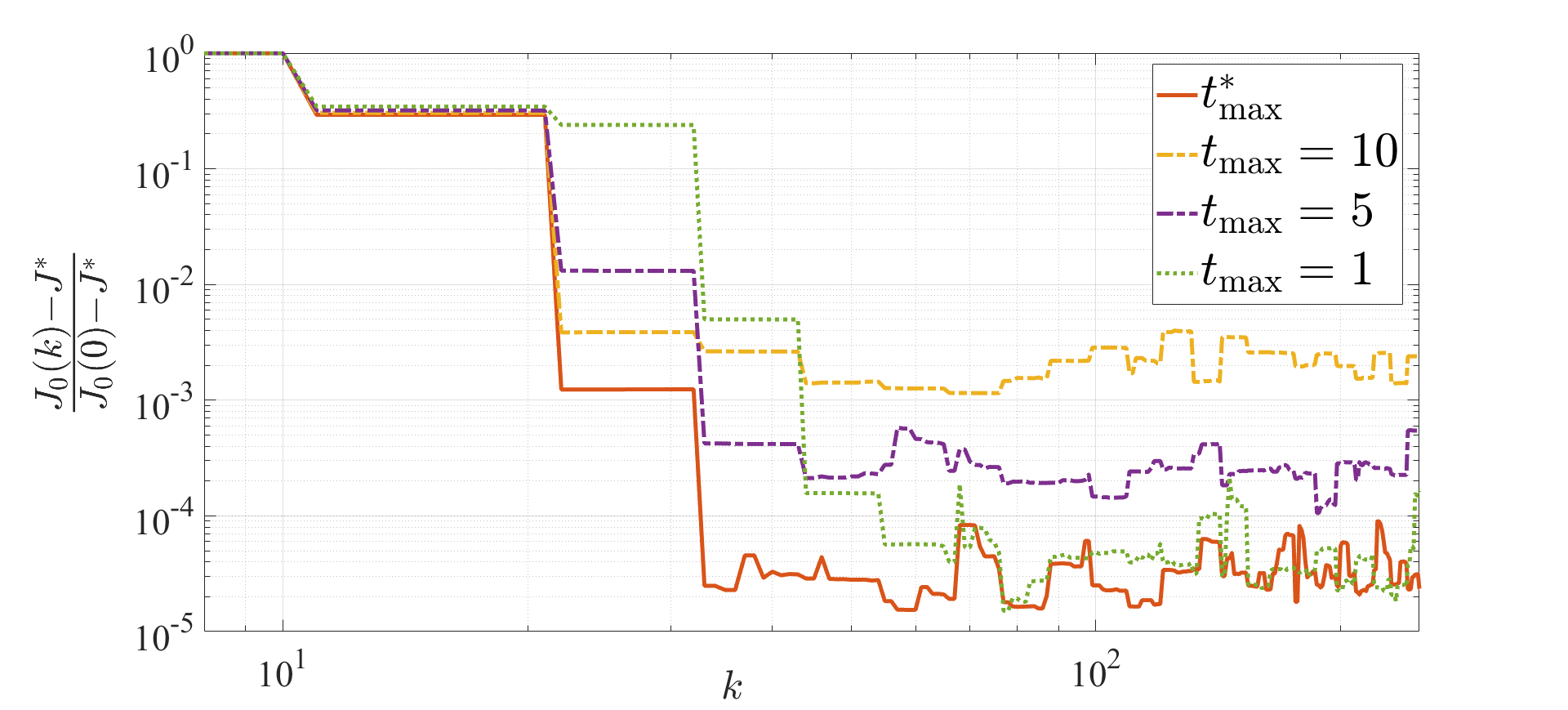}
        \captionsetup{width=\columnwidth}
        \caption{Convergence of the DSO objective, $J_0$, for different values of $t_{\max}$, with the true feasible set, $C$, and $K=10$.}
        \label{FIG3}
\end{figure}

We first run the simulations under the assumption that the DSO knows the feasible set of the EC, $C$, and the optimal number of iterations of Algorithm \ref{alg:DCP}, $t^*_\mathrm{max}$, computed as in \eqref{Crit}, which allows to maximally exploit the learned surrogate model within the trust region. In Fig. \ref{FIG1}, in the top figure, we compare the performance of our proposed method for different values of $K$ (i.e., the number of iterations at each round of Algorithm \ref{alg:ZOM}). In particular, we can see that the employment of the surrogate model remarkably accelerates the convergence, leading the normalized objective function of the DSO, $J_0$, to reach an error below $10^{-4}$ in around $200$ iterations under all scenarios. Specifically, the case of $K=40$ reaches the best accuracy in around $100$ iterations, as it largely benefits from the employment of the surrogate model and exhibits large jumps. The case of $K=10$ converges faster (in around $60$ iterations), as the jumps due to the estimation occur more often, even if they are smaller. Interestingly, the case of $K=20$ is the slowest. Finally, we notice that adapting $K$ based on the adapting rule in \eqref{AdaptiveN} leads to slower convergence due to the conservativeness of the bound. Moreover, we notice that, close to convergence, alternating zeroth-order method with estimation step with a fixed $K$ can induce oscillations, while using adaptive $K$ does not. Furthermore, the bottom figure in Fig. \ref{FIG1} shows the performance of Algorithm \ref{alg:ZOM}, when it is not assisted by the estimation of the user response, revealing a significant reduction in the number of required iterations between our proposed learning-based method and Algorithm \ref{alg:ZOM}. 

Next, we relax the assumption about full knowledge of the feasible set $C$ of the EC, defined in Lemma \ref{lemma1}. In particular, we assume that the DSO has access only to the matrix $G$, which only depends on the type and the number of assets with which the EC is equipped. As such, vector $\mathbf{h}$, containing individual parameters and private device specifications, is approximated from data, such that $
G \xbf_k \le \hat{\mathbf{h}} \quad \forall \, (\ybf_k, \xbf_k) \in \mathcal{D}.$ In Fig. \ref{FIG2}, the convergence behavior of our proposed method, using the approximation $\hat{\mathbf{h}}$ rather than the true one, is shown. We can observe that the method is slower, as it takes around $700$ iterations, and the reached accuracy slightly deteriorates to $10^{-3}$. Yet, we remark that this still quite outperforms the typical accuracy reached by Algorithm \ref{alg:ZOM}, when it is not assisted by the estimation of the user response. 

Finally, Fig. \ref{FIG3} shows the impact of properly tuning $t_\mathrm{max}$ in Algorithm \ref{alg:DCP} for the case with $K=10$, and true $\mathbf{h}$. We highlight that using a large value of $t_\mathrm{max}$, even if it might lead to larger jumps for early iterations, deteriorates the performance for later iterations. In fact, the case $t_\mathrm{max}=10$ cannot reach a level of accuracy below $10^{-3}$, since the effect of the surrogate model from the $800$-th iteration onward is disruptive, leading to upwards jumps (indicating that the iterations of Algorithm \ref{alg:DCP} violated the trust region) which cause divergence. On the other side, small values of $t_\mathrm{max}$, such as $t_\mathrm{max}=1$, by consistently leading to smaller jumps, guarantee that the trust region is always satisfied, which renders the method slower, but guarantees its convergence (with accuracy $10^{-4}$). 
\section{Conclusion}\label{Conclusion}
In this paper, a hybrid learning-based zeroth-order framework for the design of dynamic network tariffs in DSM is introduced. In departure from existing zeroth-order methods that treat the EC as a black box, we considered the case in which the DSO knows the parametric structure of the user reactions to a given price. As such, we proposed an approach to efficiently alternate zeroth-order updates with data-driven estimation of the lower-level response to improve the efficiency of the convergence to an equilibrium tariff. Our numerical experience, based on real-world consumption data, showed that this integration can significantly (i.e., several orders of magnitude) accelerate convergence compared to standard zeroth-order methods. In the context of DSM, this translates into a reduced number of interactions between the DSO and the EC, which is a key requirement for practical implementation.

\bibliographystyle{IEEEtran}
\bibliography{biblio}

@article{nesterov2017random,
  title={Random gradient-free minimization of convex functions},
  author={Nesterov, Yurii and Spokoiny, Vladimir},
  journal={Foundations of Computational Mathematics},
  volume={17},
  number={2},
  pages={527--566},
  year={2017},
  publisher={Springer}
}

@inproceedings{maheshwari2024follower,
  title={Follower agnostic learning in {S}tackelberg games},
  author={Maheshwari, Chinmay and Cheng, James and Sastry, Shankar and Ratliff, Lillian and Mazumdar, Eric},
  booktitle={63rd IEEE Conference on Decision and Control (CDC)},
  pages={222--228},
  year={2024},
  organization={IEEE}
}

@article{askeland2023stochastic,
  title={A stochastic {MPEC} approach for grid tariff design with demand-side flexibility},
  author={Askeland, Magnus and Burandt, Thorsten and Gabriel, Steven A},
  journal={Energy systems},
  volume={14},
  number={3},
  pages={707--729},
  year={2023},
  publisher={Springer}
}

@article{koolen2017machine,
  title={Machine learning for identifying demand patterns of home energy management systems with dynamic electricity pricing},
  author={Koolen, Derck and Sadat-Razavi, Navid and Ketter, Wolfgang},
  journal={Applied Sciences},
  volume={7},
  number={11},
  pages={1160},
  year={2017},
  publisher={MDPI}
}

@article{guo2019drivers,
  title={Drivers of domestic electricity users’ price responsiveness: A novel machine learning approach},
  author={Guo, Peiyang and Lam, Jacqueline CK and Li, Victor OK},
  journal={Applied energy},
  volume={235},
  pages={900--913},
  year={2019},
  publisher={Elsevier}
}

@book{luo1996mathematical,
  title={Mathematical programs with equilibrium constraints},
  author={Luo, Zhi-Quan and Pang, Jong-Shi and Ralph, Daniel},
  year={1996},
  publisher={Cambridge University Press}
}

@article{zhang2013robust,
  title={Robust energy management for microgrids with high-penetration renewables},
  author={Zhang, Yu and Gatsis, Nikolaos and Giannakis, Georgios B},
  journal={IEEE {T}ransactions on {S}ustainable {E}nergy},
  volume={4},
  number={4},
  pages={944--953},
  year={2013},
  publisher={IEEE}
}

@article{nguyen2017optimal,
  title={Optimal demand response and real-time pricing by a sequential distributed consensus-based ADMM approach},
  author={Nguyen, Dinh Hoa and Narikiyo, Tatsuo and Kawanishi, Michihiro},
  journal={IEEE Transactions on Smart Grid},
  volume={9},
  number={5},
  pages={4964--4974},
  year={2017},
  publisher={IEEE}
}

@article{hutchinson2024safe,
  title={Safe pricing mechanisms for distributed resource allocation with bandit feedback},
  author={Hutchinson, Spencer and Turan, Berkay and Alizadeh, Mahnoosh},
  journal={IEEE Transactions on Control of Network Systems},
  volume={11},
  number={4},
  pages={2010--2021},
  year={2024},
  publisher={IEEE}
}

@article{lu2018dynamic,
  title={A dynamic pricing demand response algorithm for smart grid: Reinforcement learning approach},
  author={Lu, Renzhi and Hong, Seung Ho and Zhang, Xiongfeng},
  journal={Applied {E}nergy},
  volume={220},
  pages={220--230},
  year={2018},
  publisher={Elsevier}
}

@INPROCEEDINGS{cianchi2025two,
  title={A two-part pricing mechanism for demand side management},
  author={Cianchi, Silvia and Sanjab, Anibal and Grammatico, Sergio},
  booktitle={11th IEEE International Conference on Control, Decision and Information Technologies (CoDIT)},
   year={2025},
  pages={2245--2250}
}

@ARTICLE{SWEDEN,
  author={Ruwaida, Yvonne and Chaves-Avila, José Pablo and Etherden, Nicholas and Gomez-Arriola, Ines and Gürses-Tran, Gonca and Kessels, Kris and Madina, Carlos and Sanjab, Anibal and Santos-Mugica, Maider and Trakas, Dimitris N. and Troncia, Matteo},
  journal={IEEE Transactions on Power Systems}, 
  title={{TSO-DSO}-Customer Coordination for Purchasing Flexibility System Services: Challenges and Lessons Learned from a Demonstration in {S}weden}, 
  year={2023},
  volume={38},
  number={2},
  pages={1883-1895}
}

@article{facchinei2003finite,
  title={Finite-Dimensional Variational Inequalities and Complementarity Problems},
  author={Facchinei, Francisco and Pang, JONG SHI},
  year={2007},
  publisher={Springer Science and Business
Media}
}

@article{grontas2024big,
  title={Big hype: Best intervention in games via distributed hypergradient descent},
  author={Grontas, Panagiotis D and Belgioioso, Giuseppe and Cenedese, Carlo and Fochesato, Marta and Lygeros, John and D{\"o}rfler, Florian},
  journal={IEEE Transactions on Automatic Control},
  volume={69},
  number={12},
  pages={8338--8353},
  year={2024},
  publisher={IEEE}
}

@INPROCEEDINGS{10886163,
  author={Liang, Zhirui and Li, Qi and Comden, Joshua and Bernstein, Andrey and Dvorkin, Yury},
  booktitle={63rd IEEE Conference on Decision and Control (CDC)}, 
  title={Learning with Adaptive Conservativeness for Distributionally Robust Optimization: Incentive Design for Voltage Regulation}, 
  year={2024},
  volume={},
  number={},
  pages={866-873},
  doi={10.1109/CDC56724.2024.10886163}}

@inproceedings{alahmed2024network,
  title={Network-aware and welfare-maximizing dynamic pricing for energy sharing},
  author={Alahmed, Ahmed S and Cavraro, Guido and Bernstein, Andrey and Tong, Lang},
  booktitle={63rd IEEE Conference on Decision and Control (CDC)},
  pages={859--865},
  year={2024},
  organization={IEEE}
}

@book{von2010market,
  title={Market structure and equilibrium},
  author={Von Stackelberg, Heinrich},
  year={1934},
  publisher={Springer}
}

@article{albadi2008summary,
  title={A summary of demand response in electricity markets},
  author={Albadi, Mohamed H and El-Saadany, Ehab F},
  journal={Electric power systems research},
  volume={78},
  number={11},
  pages={1989--1996},
  year={2008},
  publisher={Elsevier}
}

@INPROCEEDINGS{9993196,
  author={Fochesato, Marta and Cenedese, Carlo and Lygeros, John},
  booktitle={61st IEEE Conference on Decision and Control (CDC)}, 
  title={A {S}tackelberg game for incentive-based demand response in energy markets}, 
  year={2022},
  volume={},
  number={},
  pages={2487-2492},
  doi={10.1109/CDC51059.2022.9993196}}

@book{dempe2002foundations,
  title={Foundations of bilevel programming},
  author={Dempe, Stephan},
  year={2002},
  publisher={Springer}
}

@article{tao1997convex,
  title={Convex analysis approach to {DC} programming: theory, algorithms and applications},
  author={Tao, Pham Dinh and An, LT Hoai},
  journal={Acta mathematica vietnamica},
  volume={22},
  number={1},
  pages={289--355},
  year={1997}
}

@article{nesterov2005smooth,
  title={Smooth minimization of non-smooth functions},
  author={Nesterov, Yu},
  journal={Mathematical programming},
  volume={103},
  number={1},
  pages={127--152},
  year={2005},
  publisher={Springer}
}

@book{clarke1990optimization,
  title={Optimization and nonsmooth analysis},
  author={Clarke, Frank H},
  year={1990},
  publisher={SIAM}
}

@article{li2025new,
  title={New penalized stochastic gradient methods for linearly constrained strongly convex optimization},
  author={Li, Meng and Grigas, Paul and Atamt{\"u}rk, Alper},
  journal={Journal of Optimization Theory and Applications},
  volume={205},
  number={2},
  pages={29},
  year={2025},
  publisher={Springer}
}

\end{document}